\documentclass[11pt]{article}

\usepackage{geometry}
\usepackage{amssymb}
\usepackage{amsthm}
\usepackage{graphicx}

\usepackage{hyperref}
\usepackage{authblk}
\usepackage{psfrag}
\usepackage{enumitem}

\newcommand{\N}{\mathbb{N}}
\newcommand{\Z}{\mathbb{Z}}

\newcommand{\R}{\mathbb{R}}

\usepackage{mathtools}

\newcommand{\VT}{\mathcal V}
\newcommand{\Wave}{\mathcal W}

\newcommand{\Ao}{\mathbf{A}}
\newcommand{\wave}{\mathbf{U} }
\newcommand{\Mc}{\mathcal{M}}

\newcommand{\Mo}{\mathbf{M}}
\newcommand{\Do}{\mathbf{D}}

\newcommand{\Ko}{\mathbf{K}}
\newcommand{\So}{\mathbf{S}}
\newcommand{\Po}{\mathbf{P}}

\newcommand{\wscale}{\phi}
\newcommand{\wwave}{\psi}
\newcommand{\V}{\mathrm V}
\newcommand{\W}{\mathrm W}
\newcommand{\pr}{\mathcal P}
\newcommand{\sph}{\mathbb S}

\newcommand{\wn}{Z}
\newcommand{\la}{\lambda}
\newcommand{\soft}{\mathbf s}
\newcommand{\Soft}{\mathbf S}
\newcommand{\data}{g}
\newcommand{\noise}{z}
\newcommand{\threshp}{w}
\newcommand{\edot}{\,\cdot\,}
\newcommand{\ew}{\mathbb E}

\newcommand{\signal}{f}
\newcommand\B{\mathcal{B}}

\newcommand{\eps}{\epsilon}

\newcommand{\La}{\Lambda}
\newcommand{\risk}{\boldsymbol{\Delta}}
\newcommand{\tik}{\boldsymbol{\Phi}}
\newcommand{\est}{\mathbf{R}}
\newcommand{\rf}{\mathcal{J}}
\newcommand{\half}{\mathbb{H}}

\newcommand{\zz}{z}
\newcommand{\xx}{x}
\newcommand{\yy}{y}

\newcommand{\fn}{{\tt f}}
\newcommand{\gn}{{\tt g}}
\newcommand{\cn}{{\tt c}}
\newcommand{\vn}{{\tt v}}
\DeclareMathOperator{\An}{\tt A}
\DeclareMathOperator{\Bn}{\tt B}
\DeclareMathOperator{\Vn}{\tt V}
\DeclareMathOperator{\Wn}{\tt W}

\newcommand{\kl}[1]{\left(#1\right)}       
\newcommand{\inner}[1]{\left\langle#1\right\rangle}      
\newcommand{\norm}[1]{\left\lVert#1\right\rVert}      
\newcommand{\abs}[1]{\left\lvert#1\right\rvert}                 
\newcommand\set[1]{\left\{#1\right\}}

\newcommand{\sinner}[1]{\langle#1\rangle}      
\newcommand{\snorm}[1]{\lVert#1\rVert}      
\newcommand{\sabs}[1]{\lvert#1\rvert}                 
\newcommand{\sparen}[1]{\left\{#1\right\}}		      
\newcommand{\ran}[1]{\mathrm{ran}\left(#1\right)}     
\newcommand{\dom}[1]{\mathrm{dom}\left(#1\right)}     

\renewcommand{\d}{\,\mathrm{d}}						  

\DeclareMathOperator*{\argmin}{\mathrm{arg\:min}}

\renewcommand{\epsilon}{\varepsilon}
\renewcommand{\rho}{\varrho}
\renewcommand{\phi}{\varphi}

\setenumerate{label=(\alph*),topsep=0em,leftmargin=4ex}
\setitemize{label=\scriptsize{$\blacksquare$},topsep=0em}

\usepackage{algorithmicx}
\usepackage[ruled]{algorithm}
\usepackage{algpseudocode}

\newtheorem{theorem}{Theorem}
\newtheorem{lemma}[theorem]{Lemma}

\newtheorem{definition}[theorem]{Definition}

\title{Efficient regularization with wavelet sparsity constraints in PAT}

\author{J\"urgen Frikel} 
\affil{Department of Computer Science and Mathematics\\
Universit\"atsstra{\ss}e 31, D-93053 Regensburg, Germany\\
E-mail: {\tt ‎juergen.frikel@oth-regensburg.de}} 

\author{Markus Haltmeier} 
\affil{Department of Mathematics, University of Innsbruck\\
Technikestra{\ss}e 13, A-6020 Innsbruck, Austria\\
E-mail: {\tt markus.haltmeier@uibk.ac.at}}

\date{March 23, 2017}

\begin{document}
\maketitle

\begin{abstract}
In this paper we consider the reconstruction problem of photoacoustic tomography (PAT) with a flat observation surface. We develop a direct reconstruction method that employs regularization with wavelet sparsity constraints. To that end, we derive a wavelet-vaguelette decomposition (WVD) for the PAT  forward operator and a corresponding explicit
reconstruction formula in the case of exact data.  In the case of noisy data, we  combine the WVD reconstruction formula with soft-thresholding which  yields a spatially adaptive estimation method. We demonstrate that our method is statistically optimal for white random noise if the unknown function is assumed to lie in any Besov-ball. We present generalizations of this approach and, in particular, we discuss the combination of vaguelette soft-thresholding with a TV prior. We also provide an efficient implementation of the vaguelette transform that leads to fast image reconstruction algorithms supported by numerical results.



\end{abstract}

\section{Introduction}
\label{sec:intro}

Photoacoustic tomography (PAT) is a novel coupled-physics  (hybrid) mo\-dality for non-invasive biomedical imaging that combines the  high contrast of optical tomography with the high spatial resolution of acoustic imaging \cite{beard2011biomedical,kruger1995photoacoustic,wang2009multiscale,wang2011photoacoustic,xu2006photoacoustic}. Its principle consists in illuminating a sample by an electromagnetic pulse that, due to the photoacoustic effect, generates pressure waves inside of the object; see Figure~\ref{fig:pat}. The generated pressure waves (the acoustic signals) then propagate through the sample and beyond, and the pressure is recorded  outside of the sample.  Finally, mathematical
algorithms are used to reconstruct  an image of the interior (see, for example, \cite{kuchment2011mathematics,paltauf2009photoacoustic,xu2006photoacoustic}).

In this paper we work with the standard model of PAT, where the acoustic pressure $u \colon \R^d \times \kl{0, \infty} \to \R$ solves the  standard wave equation (see \eqref{eq:wave1}). The goal of PAT is to recover the initial pressure distribution (at time $t=0$), given by a function $h:\R^d\to\R$, from measurements $\wave h \coloneqq u|_{ \partial \half_+ \times (0, \infty) }$ of the pressure that is recorded on the hyperplane $\partial \half_+$. Here, $\half_+$ denotes the half space $\R^{d-1} \times (0, \infty)$. That is, we assume that we know $u$ on $\partial \half_+ = \R^{d-1}\times\{0\}$ (the acquisition surface), and our goal is to reconstruct $h$ from  (possibly approximate) knowledge of $\wave h $. This is an inverse problem,  which amounts to an (approximate) inversion of the operator $\wave$. We will refer to that problem as the inverse problem of PAT with a planar acquisition geometry. Note that similar approaches can be considered for other acquisition geometries as well.

\begin{figure}[tbh!]
\centering
  \includegraphics[width=\textwidth]{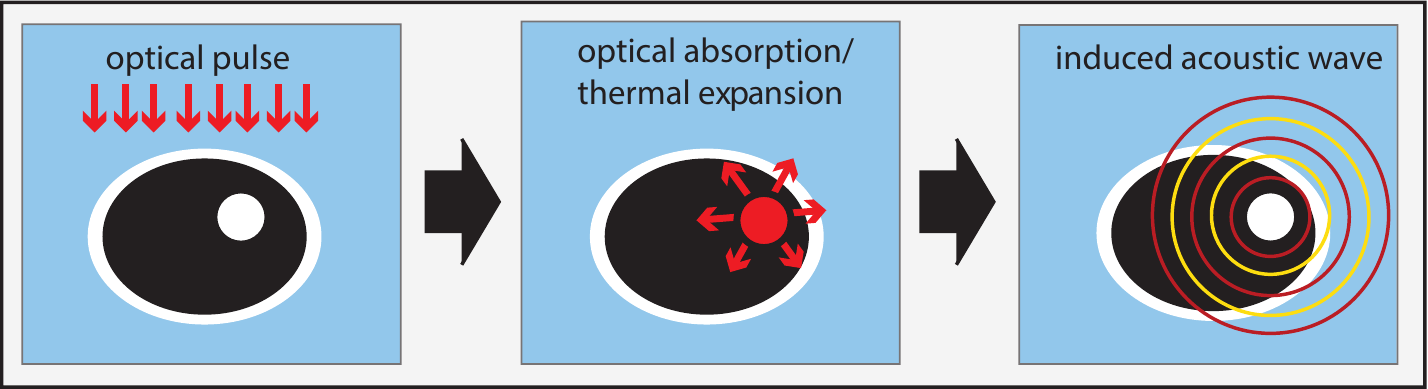}
\caption{\textsc{Basic principle of PAT.}\label{fig:pat}  A semitransparent sample is illuminated with a  short optical pulse. Due to optical absorption and subsequent thermal expansion an acoustic pressure wave is induced within  the sample. The pressure waves are measured outside of the sample and used to reconstruct an image of the interior.}
\end{figure}

In the case of exact data $\wave h$,  several approaches have been derived for solving the  inverse problem of PAT considering different acquisition geometries. This includes time reversal (see  \cite{burgholzer2007exact,Hristova2008,FinPatRak04,nguyen2016dissipative,Treeby10}), Fourier domain  algorithms (see \cite{HalSchZan09b,AgrKuc07,Kun07b,xu2002exact1,KoeFraNiePalWebFre01,JaeSchuGerKitFre07,BK78,NorLin81,Faw85,And88}), analytic reconstruction formulas  of back-projection
type (see \cite{palamodov2012uniform,FinHalRak07,FinPatRak04,haltmeier13inversion,haltmeier14universal,HalPer15b,Kun07a,natterer2012photo,xu2005universal}), as well as iterative
approaches \cite{DeaBueNtzRaz12,haltmeier2016iterative,PalNusHalBur07b,PalViaPraJac02,RosNtzRaz13,zhang2009effects,wang2014discrete,
wang2012investigation}. For the case of noisy data, it is well known that iterative methods (including variational methods and sparse reconstructions) tend to be more accurate than analytic methods.

Among those iterative techniques, sparse regularization approaches have gained a lot of attention during the last years as they have proven to perform well for noisy data as well as for incomplete data problems. One of their main advantages consists in their ability to combine efficient regularization with good feature preservation and to (to some extent) to compensate for the missing data \cite{betcke2016acoustic,frikel2013sparse,provost2009application,haltmeier2016compressed}. However, these advantages of sparse regularization methods come with the cost of typically significantly longer reconstruction times than FBP-type approaches. This is because the forward and adjoint operators have to be evaluated repeatedly. Due to that reason, FBP
type methods (or other direct approaches)  are often preferred over the more elaborate sparse regularization techniques \cite{agranovsky2009reconstruction,arridge2016adjoint,burgholzer2007exact,pan2009why,poudel2017mitigation}.

In this paper we develop numerically efficient reconstruction method for PAT with planar geometry that effectively deals with noisy data $g=\wave h+\noise$, where regularization is achieved by enforcing sparsity constraints in the reconstruction with respect to wavelet coefficients of $h$. More precisely, we derive a direct method for calculating a minimizer the $\ell^1$-Tikhonov
functional
\begin{equation} \label{eq:tik1}
	\tik_{\data,\threshp}(h)
	\coloneqq
	\frac{1}{2}
	\norm{\wave  h - \data}^2
	+
	 \sum_{\la \in \La}
	 \threshp_{\la} \abs{\inner{\wwave_\la, h}},
\end{equation}
where $(\psi_\la)_{\la \in \La}$ denotes an orthonormal wavelet basis and $w_\la >0$ are weights. To achieve that, as one of our main results, we construct a wavelet-vaguelette decomposition for the operator $\wave$. That is, given a wavelet basis $(\psi_\la)_{\la \in \La}$, we construct a system $(v_\lambda)_{\la \in \La} $ in $\ran{\wave}$ that, in the case of exact data $g=\wave h$, gives rise to an inversion formula of the form
$h = \sum_{\lambda} \kappa_\lambda \langle \wave h, v_\lambda\rangle \psi_\lambda$.
Given noisy data $g=\wave h+\noise$, we show that a combination of this formula with soft-thresholding $\soft_\threshp$, namely
\begin{equation}
\label{eq:wvd rec}
	h^* = \sum_{\lambda} \kappa_\lambda \, \soft_\threshp( \langle g, v_\lambda\rangle) \psi_\lambda,
\end{equation}
provides a minimizer of the functional \eqref{eq:tik1}. Additionally, we derive an efficient algorithm for the vaguelette transform $g\mapsto (\langle g, v_\lambda\rangle)_\lambda$ and provide an implementation for the WVD reconstruction \eqref{eq:wvd rec}.  Moreover, we show order optimality of our method in the case of  deterministic noise as  well as is the case of Gaussian random noise.  We also consider generalizations of our method and, in particular, we show how WVD reconstruction can be combined with an additional TV-prior.

Sparse regularization has been widely used as a reconstruction method for general inverse problems and there is a vast literature on that topic
(see, for example, \cite{burger2013convergence,daubechies2004iterative,flemming2010new,grasmair2008sparse,grasmair2011necessary,haltmeier2013stable,
lorenz2008convergence,ramlau2006tikhonov,scherzer2009variational,vaiter2013robust}).
In the statistical setting, $\ell^1$-Tikhonov regularization is known as LASSO  \cite{Tib96,bickel2009simultaneous,Zou06,osborne2000lasso} or
basis pursuit \cite{CheDonSau01}.   In most cases, the reconstructions are computed by employing iterative algorithms (such as iterative soft-thresholding) to minimize the $\ell^1$-Tikhonov functional \cite{bredies2008,CombPes11,ComWaj05,daubechies2004iterative,figueiredo2003algorithm,teboulle2009fast}. As mentioned above, those methods have the disadvantage to be slower than FBP type methods, as the forward and adjoint problem  have to be  solved repeatedly.

Wavelet-vaguelette decompositions and generalizations like biorthogonal curvelet
decompositions and shearlet decompositions have been derived for the classical Radon transform in 2D, see \cite{AbrSil98,Donoho:1995gp,CanDon02,colonna2010radon,kolaczyk1996wavelet}.  To the best of our knowledge, this paper is the first to provide a WVD for photoacoustic tomography as well
as an efficient direct implementation of sparse regularization using wavelets for that case.

\paragraph{Organization of the paper}
In section~\ref{sec:pat} we  review the mathematical principles of PAT with a flat observation surface and collect results required for
our further analysis. In Section~\ref{sec:wvd} we derive the WVD
for the forward operator. The WVD is then used to define the soft-thresholding 
estimator in Section~\ref{sec:noisy}. In that section we also
discuss the equivalence to variational estimation such as $\ell^1$-Tikhonov
regularization. The efficient implementation of the estimator requires and
efficient implementation of the vaguelette transform. Such an algorithm is
derived  in Section~\ref{sec:num}, where   we also present results  of our numerical simulation

\section{PAT with a flat observation surface}
\label{sec:pat}

Let $C_{0}^{\infty}(\half_+)$ denote  the space of compactly supported functions $f \colon \R^d  \to \R$
that are supported in the half space  $\half_+ \coloneqq\R^{d-1}\times(0,\infty)$, where $d\geq 2$. We write $(\xx,\yy)\in\R^{d-1}\times\R=\R^{d}$ and  consider the initial value problem
\begin{equation} \label{eq:wave1}
\begin{array}{rll}
	(\partial_{t}^{2} - \Delta_{(\xx,\yy)}) u(\xx,\yy,t) &= 0, \quad &(\xx,\yy,t)\in\R^{d}\times(0,\infty),\\
	u(\xx,\yy,0) &= h(\xx,\yy), &(\xx,\yy) \in\R^{d},\\
	\partial_{t}u(\xx,\yy,0) &= 0, & (\xx,\yy)\in\R^{d},
\end{array}
\end{equation}
with $f\in C_{0}^{\infty}(\half_+)$. 
We assume that the pressure data is observed on the hyperplane $\partial \half_+ = \R^{d-1}\times\{0\}$ and define the corresponding PAT forward operator as
\begin{align*}
	&\wave \colon C_{0}^{\infty}(\half_+)\to C^{\infty}(\R^{d})\\
	&(\wave h)(\xx,t)
	\coloneqq
	\begin{cases}
		u(\xx,0,t), & \text{for $t>0$},\\
		0,           &\text{else},
	\end{cases}
\end{align*}
where $u$ denotes the unique solution of~\eqref{eq:wave1}.
Our aim is to  recover $f$ from exact or approximate knowledge of
$\wave h$.
We are particularly interested in the cases $d=2$ and $d=3$, as they are of
practical relevance in  PAT (see \cite{kuchment2011mathematics,BurBauGruHalPal07}).
Nevertheless, in what follows, we consider the case of general dimension since this does not introduce
additional difficulties.

\subsection{Isometry property}

 The following isometry property for the wave equation is central in the analysis
 we derive below. For odd dimensions, it has been obtained in~\cite{BK78}, and for even dimensions in~\cite{Narayanan:2010cy}; see also \cite{Bel09}.

\begin{lemma}[Isometry property for the operator $\wave$]
\label{lem:isometry}
	For any $f\in C_{0}^{\infty}(\half_+)$ we have
	\begin{equation}
		\int_{\R^{d-1}}\int_{0}^{\infty}\frac{\abs{\wave h(\xx,t)}^{2}}{{t}}\d t\d \xx
		=
		\int_{\R^{d-1}}\int_{0}^{\infty}\frac{\abs{h(\xx,\yy)}^{2}}{{\yy}}\d \yy\d \xx .
	\end{equation}
\end{lemma}

\begin{proof}
See \cite{BK78} for $d$ odd
and~\cite{Narayanan:2010cy}  for $d$ even.
\end{proof}

From Lemma~\ref{lem:isometry} it follows that  $\wave$ extends to an isometry on the space
$\mathrm{cl}(C^{\infty}_{0}(\half_+))$ with respect to the scalar product
$\langle h,g\rangle_{0} = \int_{\R^{d-1}}\int_{0}^{\infty} \frac{h(x)\bar g(x)}{\yy}\d \yy\d \xx$.
In view of the isometry property and the desired wavelet-vaguelette decomposition,
instead of the operator $\wave$, it is more convenient to work with the modified operator
\begin{equation}\label{eq:waveP}
\begin{aligned}
	&\Ao \colon C_{0}^{\infty}(\half_+)\to C^{\infty}(\R^{d})\\
	&(\Ao f)(\xx,t) \coloneqq
	\begin{cases}
		(t^{-1/2} \circ \wave  \circ \yy^{1/2})( f)(\xx,t)   & \text{for  $t>0$},\\
		0, &\text{else}.
	\end{cases}
\end{aligned}
\end{equation}
It is not hard to see that $\Ao$ is an isometry with respect to the inner product
\begin{equation*}
\langle h,g\rangle_{L^2(\half_+)}
	=
	\int_{\R^{d-1}}
	\int_{0}^{\infty} h(x,y) \overline{g(x,y)}\d t\d \xx .
\end{equation*}
We have the following result.
\begin{lemma}[Isometry property for the operator  $\Ao$]
 \label{lem:Aiso}
Let $\Ao$ be defined as in~\eqref{eq:waveP}.
Then, the following assertions hold:
\begin{enumerate}
\item \label{lem:Aiso1}
For all $f \in C_{0}^{\infty}(\half_+)$ we have $\norm{\Ao f}_{L^2(\half_+)}
=  \norm{f}_{L^2(\half_+)}$.

\item \label{lem:Aiso2}
The operator $\Ao$ uniquely extends to an isometry $\Ao \colon L^2(\half_+)
\to L^2(\half_+)$.
\end{enumerate}
\end{lemma}

\begin{proof} \mbox{}
\ref{lem:Aiso1}
According to Lemma~\ref{lem:isometry} for every $f \in C_{0}^{\infty}(\half_+)$ we have
\begin{align*}
	\norm{\Ao f}_{L^2(\half_+)}^2
	&=
	\int_{\R^{d-1}}
	\int_{0}^{\infty} \sabs{\Ao f(\xx,t)}^{2}\d t\d \xx
\\	&=
	\int_{\R^{d-1}}
	\int_{0}^{\infty} \sabs{(t^{-1/2}\, \wave \, \yy^{1/2})f(\xx,t)}^{2}\d t\d \xx
	\\
	&=
	\int_{\R^{d-1}}\int_{0}^{\infty}\frac{\sabs{ \wave (\yy^{1/2} f) (\xx,t)}^{2}}{{t}}\d t\d \xx
	\\
	&=
	\int_{\R^{d-1}}\int_{0}^{\infty}\frac{\sabs{\yy^{1/2} f(\xx,\yy)}^{2}}{{\yy}}\d \yy\d \xx .
	\\
	&=
	\int_{\R^{d-1}}\int_{0}^{\infty} \sabs{f(\xx,\yy)}^{2} d \yy\d \xx .
	\\
	&=
	\norm{f}_{L^2(\half_+)}^2.
\end{align*}

\ref{lem:Aiso2} The set $C_{0}^{\infty}(\half_+)$ is
dense in $L^2(\half_+)$ and, according to \ref{lem:Aiso1}, the operator
$\Ao$ is an isometry on $C_{0}^{\infty}(\half_+)$. Consequently, it follows from the general
Hilbert space theory that the operator $\Ao$ can be  extended in a unique manner to an
isometry $\Ao \colon L^2(\half_+) \to L^2(\R^d)$. Finally, from the construction of
$\Ao $ it is clear that $\Ao \kl{L^2(\half_+)} \subseteq L^2(\half_+)$.
\end{proof}

\subsection{Isometric extension to $L^2(\R^d)$}
For the following considerations it will be convenient to apply the operator $\Ao$ to functions that are defined on $\R^d$ rather than on the half space $\half_+$. That is,
we need to extend the operator $\Ao \colon L^2(\half_+)\to L^2(\half_+)$ in a meaningful way to an operator
$\Ao\colon L^2( \R^d) \to  L^2(\R^d)$. One possibility to do this would be  to consider the wave
equation \eqref{eq:wave1} with initial data $f\in C_{0}^{\infty}(\R^d)$ and then to
proceed as above. However,  any function that is odd in the last variable would be
in the kernel of the resulting operator.  Therefore, we use a different extension
that leads to an isometric operator on $L^2(\R^d)$.

To that end, we define  the operator $\So \colon L^2(\R^d)  \to L^2(\R^d)$
by $(\So f)(\xx,\yy) \coloneqq (\So f)(\xx,-\yy)$.
Then $\So$ is an isometric isomorphism with  $\So^{-1} = \So$  and
$\So(L^2(\half_+)) = L^2(\half_-)$, where $\half_-$ is defined in an obvious way. We are now able to define the announced extension of $\Ao$.

\begin{definition} \label{def:Afull}
We define the operator $\Ao\colon L^2( \R^d) \to  L^2(\R^d)$
by
\begin{equation}
\Ao( f )  \coloneqq
\Ao \kl{ \Po_{L^2(\half_+)} f} +
( \So \circ \Ao  \circ \So) \kl{ \Po_{L^2(\half_-)} f } \,.
\end{equation}
Here and below $\Po_{V} f$ denotes the  orthogonal projection onto a closed
subspace $V$ of $L^2(\R^d)$.
\end{definition}

\begin{theorem} \label{thm:isofull}
The operator $\Ao\colon L^2(\R^d) \to  L^2(\R^d)$ is an isometry.
\end{theorem}

\begin{proof}
 Any function  $f \in  L^2( \R^d) $  can  be written in the form $f =
 \Po_{L^2(\half_+)} f + \Po_{L^2(\half_-)} f$ and satisfies $\snorm{f}_{L^2}^2
 =  \snorm{\Po_{L^2(\half_+)} f}_{L^2}^2 +  \snorm{\Po_{L^2(\half_-)} f}_{L^2}^2$.
 From Definition \ref{def:Afull} and Lemma~\ref{lem:Aiso} we then conclude that
 $\snorm{\Ao f}_{L^2}^2 = \snorm{f}_{L^2}^2 $  for every $f \in  L^2( \R^d) $.
 \end{proof}

In what follows, we will also consider the operator
\begin{equation*}
	\wave \colon \dom{\wave}
	\subseteq L^2(\R^d) \to L^2(\R^d)
	\colon
	h \mapsto   \kl{ \abs{t}^{1/2} \circ  \Ao \circ \abs{\yy}^{-1/2}   } h
\end{equation*}
as a densely defined operator on $L^2(\R^d)$.
Here, $ \dom{\wave}$  is the set of all (equivalence classes of) functions
$h \colon \R^d \to \R$ such that  $\abs{\yy}^{-1/2} h \in L^2(\R^d)$.

\subsection{Explicit expressions for $\Ao$ and its dual}
In this section we will state explicit expressions for the operator $\Ao$  and its dual.
For that purpose, we consider the spherical Radon transform  $\Mo$, which is defined as follows:
\begin{equation}\label{eq:sm}
 	\forall (\xx, t) \in \partial\half_+ \times (0, \infty)
 	\colon \qquad
    (\Mo f)(\xx, r)
    :=
    \frac{1}{ \abs{\sph^{n-1}}}
    \int_{\sph^{n-1}} f(\xx + r \omega) \, dS(\omega)
    \,,
\end{equation}
where  $\sph^{n-1} := \set{ x \in \R^n \mid  \abs{x} = 1}$ denotes the unit sphere in $\R^n$ and $\abs{\sph^{n-1}}$ is its surface
measure. A simple calculation (application of  Fubini's Theorem) shows that the dual $\Mo^*$ of the operator $\Mo$ is given by
\begin{equation} \label{eq:bp}
    (\Mo^* g) (\xx, \yy) = \frac{1}{\abs{\sph_{n-1}}}
    \int_{\R^{n-1}}
    \frac{g\kl{\zz, \sqrt{\norm{\zz -  \yy}^2 + \xx^2} }}{(\norm{\zz -  \yy}^2 + \xx^2)^{(n-1)/2}}
    \, dS(\zz) \,.
\end{equation}
The operator $\Mo^* $  is called \emph{spherical backprojection operator}, because
$\kl{\Mo^*  g}(\xx,\zz)$  integrates the function $g$ over all spheres $(\zz, r)$ that pass through the point $(\xx,\zz)$.

We will also consider the (fractional) differentiation operators
\begin{equation}\label{eq:frac}
\Do_t^\mu  := \begin{cases}
 \kl{ (2t)^{-1}\partial_t }^\mu \,,  & \text{ for } \mu \in \N \,, \\
 \kl{ (2t)^{-1}\partial_t }^{\mu+1/2} \Ao \,,   &  \text{ for } \mu \in \N -1/2 \,.
  \end{cases}
  \end{equation}
The formal $L^2$ adjoints of those operators are given by $(\Do_t^\mu)^* = (-1)^\mu t \Do_t^\mu t^{-1}$
for  $\mu \in \N$ and $(\Do_t^\mu)^* = (-1)^{\mu+1/2} \Ao^*  t \Do_t^{\mu+1/2} t^{-1} $ for $\mu \in \N - 1/2$,
where
\[
    (\Ao^* g) (t)=
    \frac{2t}{\sqrt \pi} \int_{t}^\infty \frac{g(s)}{\sqrt{s^2-t^2}}\,  ds
\,,  \qquad \text{ for }
t \in (0, \infty) \,.\]

We are now able to provide explicit expressions for the operator $\Ao$  and its dual $\Ao^*$.
\begin{lemma} \mbox{}
We have
\begin{align} \label{eq:wave}
  	\forall f \in C_{0}^{\infty}(\half_+)
	\colon \quad
	&
	\Ao f
  	=
	\frac{\sqrt \pi}{2 \Gamma(n/2)}
	 \kl{ t^{-1/2}
	  \Do_t^{(n-3)/2} t^{n-2}
	 \Mo    \yy^{1/2}  \, f
	 }  \,,\\
	 \label{eq:wave-ad}
	 \forall g \in C_{0}^{\infty}(\half_+)
	\colon \quad
	&\Ao^\ast g
	 =
	 \frac{\sqrt \pi}{2 \Gamma(n/2)}
	 \kl{\yy^{1/2} \Mo^* t^{n-2}
	 \bigl (\Do_t^{(n-3)/2} \bigr)^*
	 t^{-1/2}  g} \,.
\end{align}
\end{lemma}

\begin{proof}
The identity \eqref{eq:wave} follows from the well known explicit
expression for the solution of the wave
equation (see \cite[page~682]{CouHil62} and \cite[page~80]{Eva98}).
The identity \eqref{eq:wave-ad} follows from \eqref{eq:wave} by applying
calculation rules for the adjoint.
\end{proof}

\section{Wavelet vaguelette decomposition (WVD)}
\label{sec:wvd}

In this section, for a given wavelet basis $(\psi_\lambda)_{\la \in \La}$ of $\R^d$, we construct the WVD
of the operators $\wave$ and $\Ao$ that we
defined in the previous section and prove inversion formulae for the case of exact data. To that end, we particularly will make use of the isometry relation that we proved in Section \ref{sec:pat}.

The basic idea of the WVD is to start with an orthogonal wavelet basis
and to construct  a  possibly non-orthogonal  basis  system of the image space in such a way
that the operator  and the  prior information are simultaneously (nearly) diagonalized \cite{Donoho:1995gp}. For readers convenience, we summarized some basic facts about wavelets in Appendix \ref{app:wavelets}.

\subsection{The idea of  the WVD}

Let $\Ko \colon  \dom{\Ko} \subseteq L^2(\R^d) \to L^2(\R^d)$ be a linear, not necessarily bounded, operator and let
$(\wwave_\la)_{\la \in \La}$ be an orthonormal wavelet basis of  $L^{2}(\R^{d})$.  The construction of a wavelet-vaguelette decomposition for the operator $\Ko$ amounts to finding families $(u_\la)_\la$, $(v_\la)_\la$ in $L^2(\R^d)$ satisfying the following properties:
\begin{enumerate}[label=(WVD\arabic*), itemindent =4em, leftmargin =1em]
\item Quasi-singular relations (with $\la  = (j,k,\eps)$):
\begin{align*}
	\Ko \wwave_\la &= \kappa_{j} v_\la,\\
	\Ko^{\ast} u_\la &= \kappa_{j}\wwave_\la \,.
\end{align*}

\item Biorthogonal relations:
$\langle u_\la,v_{\la'}\rangle_{L^2} = \delta_{\la,\la'}$.

\item Near-orthogonality relations:
\begin{align} \label{eq:no1}
\forall c \in \ell^2(\La) \colon \quad
\Bigl\lVert \sum_{\la \in \La} c_\la u_\la \Bigr\rVert_{L^2}&\asymp \norm{c}_{\ell_{2}} \,,
\\\label{no2}
\forall c \in \ell^2(\La) \colon \quad
\Bigl\lVert\sum_{\la \in \La}c_\la v_\la\Bigr\rVert_{L^2}&\asymp \norm{ c }_{\ell_{2}} \,.
\end{align}
Here, $ f \asymp g$ means that there are constants $A,B>0$
such that $A\,  g\leq f \leq B\,  g$.
\end{enumerate}

Such a decomposition $(\wwave_\la, u_\la, v_\la, \kappa_j)_{\la \in \La}$ (if it exists) is called a WVD for the operator $\Ko$. Given such WVD for an operator $\Ko$, one can always obtain an explicit inversion formula for the operator $\Ko$ of the form
\begin{equation}
\label{eq:wvd inversion formula}
	\forall f \in \dom{\Ko} \colon \quad
	f = \sum_{\la \in \La}
	\kappa_{j}^{-1}\langle
	\Ko f, u_\la\rangle_{L^2} \wwave_\la \,.
\end{equation}
Note the analogy between \eqref{eq:wvd inversion formula} and the SVD decomposition.
The numbers $\kappa_{j}$ depend here only on the scale parameter $j$ and have the same meaning as singular values in the SVD. Thus, $\kappa_j$ are referred to as quasi singular values. Similarly to the SVD, the decay of the quasi singular values $\kappa_j$ reflects the the ill-posedness of the inverse problem $g=\Ko f$.

A WVD decomposition has been constructed for the classical computed tomography modeled by  the two dimensional Radon transform, where $\kappa_{j} = 2^{j/2}$; see \cite{Donoho:1995gp}. In the case of the two dimensional Radon transform, a generalization of the WVD, a so-called biorthogonal curvelet decomposition was constructed in \cite{CanDon02}. In \cite{colonna2010radon}, the authors derived biorthogonal shearlet decompositions for two and three dimensional Radon transforms.

In the case of photoacoustic tomography, there are no such decompositions available so far. In the next subsection, we establish a WVD for the operators $\wave$ and $\Ao$, which serve as forward operators for  PAT with a planar observation surface, and so automatically obtain an inversion formula for exact data.

\subsection{Construction of the WVD for PAT}
Let $(\wwave_\la)_{\la\in \La}$  be a wavelet basis of $L^2(\R^d)$.
The desired wavelet-vaguelette decomposition of $\wave$ and $\Ao$ is in fact a direct  consequence of the isometry relation of  Theorem~\ref{thm:isofull}.

\begin{theorem}[The WVD for PAT] \label{thm:wvd}
For every $\la \in \La$ define $u_\la \coloneqq \Ao \wwave_\la$ and let $v_\lambda = u_\lambda$.
Then, the family $(\wwave_\la, u_\la, v_\la, 1)_{\la \in \La}$
is a  WVD for the operator $\Ao$. Moreover, we have the inversion formulae
\begin{align} \label{eq:wvd-A}
&\forall f \in L^2(\R^d) \colon
&&f = \sum_{\la \in \La} \langle \Ao f, u_\la\rangle_{L^2} \, \wwave_\la ,
\\ \label{eq:wvd-wave}
&\forall h \in \dom{\wave} \colon
&&h = \sum_{\la \in \La} \inner{\wave  h,  t^{-1/2} u_\la}_{L^2}
\, \yy^{1/2}  \wwave_ \la  .
\end{align}
\end{theorem}

\begin{proof}\mbox{}
We start with an arbitrary function
$f\in L^2(\R^d)$ and express this function in terms of a wavelet expansion  $f = \sum_{\la \in \La} \langle f,\wwave_\la\rangle_{L^2} \wwave_\la$ with respect to $(\wwave_\la)_{\la \in \La}$.
Then according to the isometry property $(u_\la)_{\la \in \La}$ is an orthonormal basis of $\ran{\Ao}$ and further $\Ao^* u_\la = \wwave_\la $.
In particular, this amounts to a WVD decomposition with $v_\la   = u_\la$ and $\kappa_j =1$. Further, \eqref{eq:wvd-A}  is a consequence
 of the   WVD.
Next, let $f =  \yy^{-1/2}  h$ be an element in $\dom{\wave}$. Then
\begin{equation*}
\inner{\Ao f, u_\la}_{L^2}
=
\inner{t^{-1/2}  \wave  h, u_\la}_{L^2}
=
\inner{\wave  h, t^{-1/2} u_\la}_{L^2} \,.
\end{equation*}
Consequently, applying \eqref{eq:wvd-A}  yields \eqref{eq:wvd-wave}.
\end{proof}

Note that the identity \eqref{eq:wvd-wave} in  Theorem~\ref{thm:wvd}
is an explicit  inversion formula for the operator $\wave$ in the spirit of a
WVD. Instead of an orthonormal wavelet basis it  uses  the family $(\yy^{1/2}\wwave_\la)_\la$,  which is non-orthogonal with respect to the $L^2$ inner product.
Restricted to functions in $L^2(\R^d)$ that vanish outside $K \coloneqq K_+ \cup \So ( K_+ )$, where $K_+ \Subset H_+$ is any compact subset,
the operator $h \mapsto \sabs{\yy}^{1/2} h$ is an isomorphism.
Then, $(\sabs{\yy}^{1/2}\wwave_\la)_{\la \in \La}$ allows to characterize the Besov norm
$\snorm{\edot}_{\B_{p,q}^r}$ of any function  that is supported in $K$.

\begin{figure}[tbh!]
\centering
\includegraphics[width=\textwidth, trim={ 1.7cm 0.5cm 1.7cm 0.3cm}]{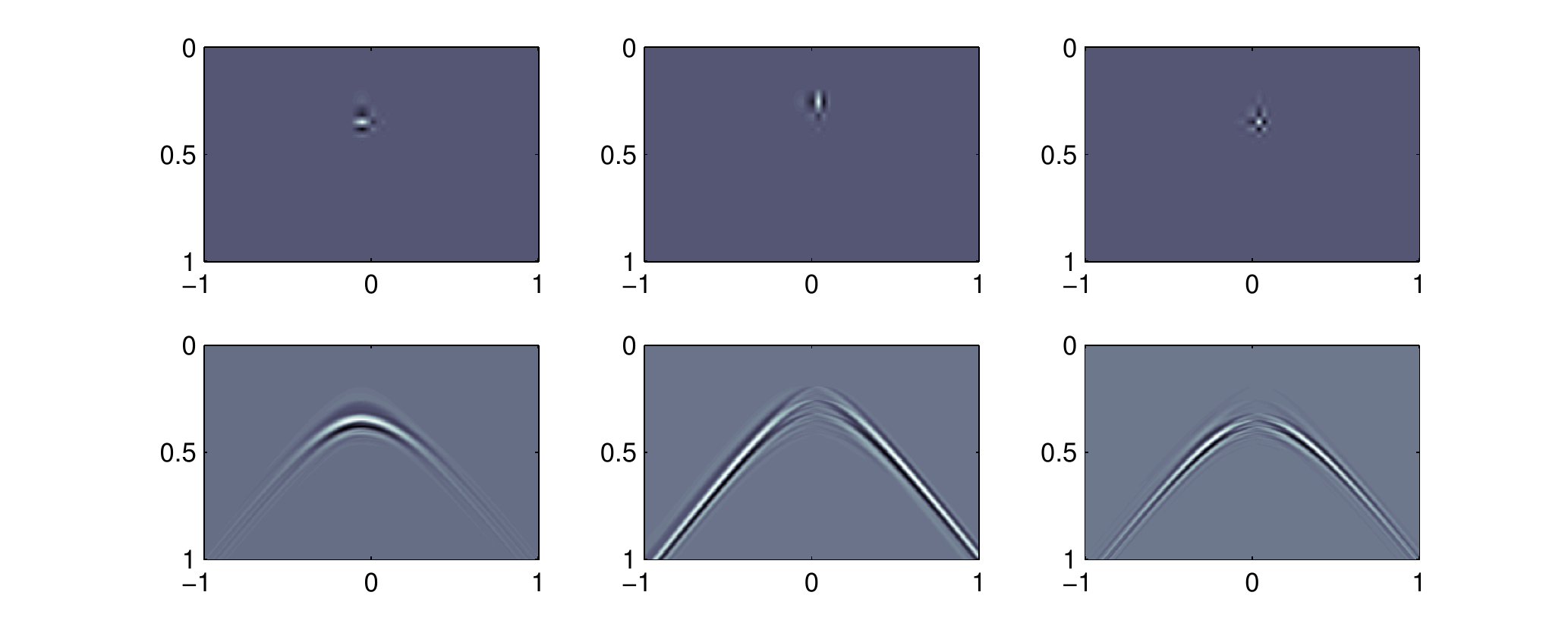}
\caption{\textsc{Wavelets and vaguelettes.}\label{fig:wv}  }
\end{figure}

Figure~\ref{fig:wv} shows a  vertical, a horizontal and a diagonal
Daubechies 10 wavelet and the corresponding vaguelettes
obtained by application of the operator $\Ao$.

\section{Inversion from noisy data}
\label{sec:noisy}
In what follows, we assume $(\wwave_\la)_{\la \in \La}$ to be an
orthonormal compactly supported wavelet  basis of $L^2(\R^d)$.
Further, by $(u_\la)_{\la \in \La}$ we denote the corresponding vaguelette basis of $\ran {\Ao} \subseteq L^2(\R^d)$ that satisfies  $\Ao \wwave_\la = u_\la$.

If exact data are available, then the WVD decomposition~\eqref{eq:wvd-A}  provides an exact reconstruction formula for the unknown $f$. However, in practical  applications the data $\Ao f$ (or $\wave h$) are only known up to some errors (e.g. noise). We therefore assume that we are given erroneous (noisy) data
\begin{equation} \label{eq:noisy}
    \data = \Ao f +  \noise \,,
\end{equation}
where $\noise$ is the error term and $f$ the exact unknown. We consider both the deterministic and the stochastic case, in which different models for $\noise$  are assumed:
\begin{itemize}
\item In the deterministic case, we assume that
a bound $\snorm{\noise}_{L^2} \leq \delta$ is available.

\item In the stochastic case, we assume that $\noise = \delta
\wn$, where  $\wn$ is a white noise process.
\end{itemize}
In the deterministic situation the constant $\delta > 0$ is referred to as the
noise level; in the stochastic case it is
the noise standard deviation.  The goal in both situations is to estimate the
unknown $f \in L^2(\R^d)$ from data in~\eqref{eq:noisy}.

\subsection{Vaguelette-thresholding estimator}

In Section~\ref{sec:wvd} we have shown that the reproducing
formula $f = \Ao^* \Ao f = \sum_{\la \in \La} \inner{\Ao f,  u_\la}_{L^2}
       \wwave_\la$ holds in the case of exact data.
If the data is corrupted by noise, i.e., $g=\Ao f +\noise$, the inner products  $\inner{\Ao f,  u_\la}_{L^2}$ cannot be computed exactly. Instead, they are estimated
by first evaluating $\inner{\data,  u_\la}_{L^2}$ and then
applying   the soft-thresholding operation
\begin{equation} \label{eq:soft}
\soft_\threshp \colon \R \to \R \colon
y \mapsto
\begin{cases}
y  + \threshp & \text{ if } y  <   -\threshp \\
0    & \text{ if  }  y \in [-\threshp, \threshp]\\
y-  \threshp & \text{ if } y  >  \threshp \\
\end{cases}
\end{equation}
with appropriate thresholds $\threshp = \threshp_j$.

\begin{definition}[WVD soft thresholding estimator] \label{def:wvd-soft}
For any nonnegative sequence  $ \threshp = (\threshp_j)_j$, the  WVD soft thresholding-estimator for the solution of
\eqref{eq:noisy} is  defined  by
\begin{equation}\label{eq:wvd-soft}
\Soft_\threshp
\colon L^2(\R^d)  \to L^2(\R^d)
\colon
g \mapsto
    \sum_\la
    \soft_{\threshp_j}(\inner{g, u_\la}_{L^2})
    \, \wwave_\la  \,,
\end{equation}
with the nonlinear soft-thresholding  function $\soft_\mu$ defined in~\eqref{eq:soft}.
\end{definition}

In the deterministic case we assume that the error term $\noise$ is known to satisfy $\snorm{\noise} \leq \delta$.
In this case,  due the isometry property of the operator $\Ao$, one obtains a stable reconstruction by applying the adjoint $\Ao^\ast$,
\begin{equation} \label{eq:astar}
 \forall f \in L^2(\R^d) \colon
 \quad
  \sup_{\snorm{\noise} \leq \delta}
\snorm{ f - \Ao^\ast (\Ao f + \noise)} = \delta  \,.
\end{equation}
As the adjoint operator satisfies $\Ao^\ast = \Soft_0$, i.e., $\Ao^*$ corresponds to the WVD-thresholding estimator with $\threshp = 0$,
it follows that also the class of WVD-thresholding estimators yields the optimal
error   estimate. Without additional knowledge about the noise there is no
need to apply the WVD soft-thresholding estimator with
$\threshp \neq 0$, even if  the function  $f$ is known
to belong to some smoothness class, such as a Sobolev
or Besov ball.   However, if a portion of the  noise energy  is known to
correspond to  high frequency components, then using
non-zero thresholds may significantly outperform
$\Ao^\ast$ given that $f$ belongs to some smoothness class.
For example, this happens in the case of
stochastic white noise,  where the energy-spectrum is
uniformly distributed.

\subsection{Optimality of vaguelette-thresholding}

Suppose $\noise  = \delta Z$, where $Z$ is a
white noise process. In the case of random noise,
it is common to measure  the performance of an estimator
$\est \colon L^2(\R^d) \to L^2(\R^d)$  in terms of the
worst-case risk~\cite{Donoho:1995gp,johnstone2015gaussian,Don95a,tsybakov2009introduction}
of a subset $\Mc \subseteq L^2(\R^d)$,
\begin{equation}\label{eq:risk}
\risk( \est, \delta, \Mc ) = \sup_{f \in \Mc}   \ew\kl{ \norm{ f - \est (\Ao f + \delta Z)}_{L^2}^2 } \,.
\end{equation}
Further define the nonlinear  minimax risk,  the minimax risk using
the WVD estimator~\eqref{eq:wvd-soft} and  the linear minimax risk,
respectively by
\begin{align*}
	\risk_{\rm N}( \delta, \Mc )
	&\coloneqq
	\inf_{\est}
	\sup_{f \in \Mc}
	\ew\kl{ \norm{ f - \est (\Ao f + \delta Z)}_{L^2}^2 }
	 \,,
	\\
	\risk_{\rm W}( \delta, \Mc)
	&\coloneqq
	\inf_{\threshp \in [0,\infty)^\N}
	\sup_{f \in \Mc}
	\ew \kl{ \norm{ f - \Soft_{\threshp} (\Ao f + \delta Z)}_{L^2}^2} \,,
	\\
	\risk_{\rm L}( \delta, \Mc)
	&\coloneqq
	\inf_{\text{$\est$ linear}}
	\sup_{f \in \Mc}
	\ew \kl{ \norm{ f - \est (\Ao f + \delta Z)}_{L^2}^2 } \,.
\end{align*}
From the definition it is clear that no  reconstruction method
$\est \colon L^2(\R^d) \to L^2(\R^d)$ can have worst-case risk  $\risk( \est, \delta, \Mc ) $   smaller than the non-linear minimax risk $\risk_{\rm N} ( \delta, \Mc )$. We are in particular interested in  asymptotic behavior for the case $\delta \to 0$ and $\Mc$ is a ball in a Besov space.

\begin{theorem} \label{thm:random}
Suppose that $r >  d(1/p - 1/2) $ and let $\Mc$ be a ball in the
Besov-norm having the form
\begin{equation}
	\Mc =  \set{ f \in L^2(\R^d) \mid \norm{f}_{\B_{p,q}^r} < \rho}
	\quad \text{ for some $\rho >0$} \,.
\end{equation}
 Then,  as $\delta \to 0$,
the following hold
\begin{enumerate}
\item\label{thm:randomN} $\risk_{\rm N}( \delta, \Mc ) \asymp \delta^{\frac{2r}{r+d/2}}$,
\item\label{thm:randomW}
$\risk_{\rm W}( \delta, \Mc )
\leq c_{\rm W} \,   \risk_{\rm N}( \delta, \Mc )$
for some constant $c_{\rm W}>0$,
\item\label{thm:randomL} $\risk_{\rm L}( \delta, \Mc ) \asymp \delta^{2 \frac{r+ d(1/2-1/p_-)}{r+ d(1-1/p_-)}}$ with $p_- \coloneqq \min \set{2,p}$.
\end{enumerate}
\end{theorem}

\begin{proof}
Follows from \cite[Theorem~4]{Donoho:1995gp}
for the special case $\kappa_j  = 1$.
\end{proof}

Theorem \ref{thm:random} implies that, despite its simplicity, the
WVD-thresholding estimator is order optimal on any Besov-ball
and the rate cannot  be improved by any
other estimator (up to some constant factors). On the other hand, if $p< 2$, then the exponent
in the linear minimax rate is strictly smaller than the  exponent  in the
non-linear minimax rate, $ \tfrac{r+ d(1/2-1/p_-)}{r+ d(1-1/p_-)} < \tfrac{r}{r+ d/2}$. Therefore, no linear estimator can give the
optimal convergence order. In particular, this implies that the
 WVD-thresholding estimator outperforms filter based regularization methods
 including the truncated SVD or quadratic Tikhonov regularization.

 \subsection{Variational characterizations and extensions}

The  WVD-based soft thresholding estimation can be
characterized via various variational minimization
schemes that, as we shall discuss later, in turn offer several extension of the
WVD-estimators.

\begin{theorem}[Variational characterizations of vaguelette thresholding]
Let $(\threshp_j)_j$\label{thm:var} be a sequence of thresholds
and  let $g, \hat f \in L^2(\R^d)$.
Then the following assertions  are equivalent:
\begin{enumerate}[label=(\arabic*)]
\item\label{thm:var1} $\hat f  = \Soft_\threshp (g)$;
\item\label{thm:var2} $\hat f  = \argmin \set{  \tfrac{1}{2}
\norm{\Ao f- g}^2 + \sum_{\la \in \La}
\threshp_j  \abs{\inner{\wwave_\la, f}} \mid f \in L^2(\R^d)}$;
\item\label{thm:var3} $\hat f$ is the unique solution of the  constraint optimization problem
  \begin{equation} \label{eq:wvd-const}
	\left\{
	\begin{aligned}
	&\min_{f \in L^2(\R^d)}  &&
	\frac{1}{2} \norm{f}_{L^2}^2 \\
	&\text{such that} &&
	\max_{\la \in \La }
	\frac{\abs{\inner{u_\la, g -  \Ao f }}}
	{\threshp_j}
	\leq   1
	\,.
	\end{aligned}
	\right.
\end{equation}
\end{enumerate}
\end{theorem}

\begin{proof} \mbox{}
$\ref{thm:var1} \Leftrightarrow \ref{thm:var2}$:
Let  $\hat f$ denote the minimizer of $\tfrac{1}{2}
\norm{\Ao f- g}^2 + \sum_{\la \in \La}
\threshp_j  \abs{\inner{\wwave_\la, f}}$.
Because $(u_\la)_{\la \in \La}$ is an orthonormal basis of
$\ran{\Ao}$ we have $\snorm{\Ao f - g}^2
=
\snorm{\pr_{\ran{\Ao}^\bot} (g)}^2 + \sum_{\la \in \La} \abs{\inner{\Ao f - g, u_\la}}^2$. Further,
$\sum_{\la \in \La}
\threshp_j  \abs{\inner{\wwave_\la, f}}
= \sum_{\la \in \La}
\threshp_j  \abs{\inner{u_\la, \Ao f}}$, which shows
 that $\hat f$ is the unique minimizer of
\begin{equation*}
\sum_{\la \in \La }
\tfrac{1}{2}
\abs{\inner{\Ao f - g, u_\la}}^2
+\threshp_j  \abs{\inner{u_\la, \Ao f}} \,.
\end{equation*}
The latter functional is minimized by minimizing  every
summand
\begin{equation*}
\Phi (\inner{\Ao f , u_\la}; \threshp_j)
=
\tfrac{1}{2}
\abs{\inner{\Ao f , u_\la} - \inner{g, u_\la}}^2
+\threshp_j  \abs{\inner{u_\la, \Ao f}}
\end{equation*}
independently in the first argument. The minimizer is given by one-dimensional soft thresholding which gives
$ \sinner{\hat f, \wwave_\la} = \sinner{\Ao \hat f, u_\la} =
\soft_{\threshp_k}(\sinner{g, u_\la})$ and therefore $\hat f  =
\Soft_\threshp (g)$.

$\ref{thm:var1} \Leftrightarrow \ref{thm:var3}$:
Let $\hat f$ denote the solution of~\eqref{eq:wvd-const}.
Using that $(\wwave_\la)_{\la \in \La}$ is an orthogonal basis
 and $\wwave_\la = \Ao^* u_\la$, shows that $\hat f$
 can be equivalently characterized as the minimizer of
  \begin{equation} \label{eq:wvd-const2}
	\left\{
	\begin{aligned}
	&\min_{f \in L^2(\R^d)}  &&
	\frac{1}{2} \sum_{\la \in \La }
	\abs{\inner{\Ao f , u_\la}}^2 \\
	&\text{such that} &&
	\max_{\la \in \La }
	\frac{\abs{\inner{u_\la, g -  \Ao f }}}
	{\threshp_j}
	\leq   1
	\,.
	\end{aligned}
	\right.
\end{equation}
The minimization  problem \eqref{eq:wvd-const2} can again be
solved separately  for  every component $\inner{\Ao f , u_\la}$ which is again given by the one-dimensional soft thresholding
$\sinner{\Ao \hat f, u_\la} =
\soft_{\threshp_j}(\sinner{g, u_\la})$. As above this implies
$\hat f  = \Soft_{\threshp} (g)$.
\end{proof}

The variational characterizations of Theorem~\ref{thm:var}
have several important implications. First, they provide an explicit
minimizer for the $\ell^1$-Tikhonov functional
\begin{equation*}
\tik_{\data,\threshp}(f) \coloneqq
\frac{1}{2}
\norm{\Ao f- g}^2 + \sum_{\la \in \La}
\threshp_j  \abs{\inner{\wwave_\la, f}},
\end{equation*}
which in general has to be minimized by an iterative algorithm,
such as the iterative soft thresholding algorithm and its
variants  \cite{CombPes11,ComWaj05,daubechies2004iterative,teboulle2009fast}.
For the analysis of $\ell^1$-Tikhonov regularization for inverse problems see \cite{daubechies2004iterative,grasmair2008sparse,grasmair2011necessary,lorenz2008convergence}.  
Another important consequence is that the WVD-soft thresholding estimator can be generalized in various directions. In particular, one can get a generalization of \eqref{eq:wvd-const} by replacing the $L^2$-norm in \eqref{eq:wvd-const} by an appropriate regularization functional $\rf$ (for example, $\rf$
can be chosen as the total variation norm, see Section~\ref{sec:implementation}). That is, instead of solving the problem  \eqref{eq:wvd-const},
one aims at solving the problem
\begin{equation} \label{eq:dantzig}
	\left\{
	\begin{aligned}
	&\min_{f \in L^2(\R^d)}  &&
	\rf(f) \\
	&\text{such that} &&
	\max_{\la \in \La }
	\frac{\abs{\inner{u_\la, g -  \Ao f }}}
	{\threshp_j}
	\leq   1
	\,,
	\end{aligned}
	\right.
\end{equation}
This generalization constitutes a hybrid version that combines the WVD estimator with more general regularization functionals $\rf$. Such hybrid approaches have been introduced independently in
\cite{CanGuo02,DurFro01,Mal02a,StaDonCan01}
(see also \cite{ChaZho00,CoiSow00}). It is also related to the
Dantzig and multiscale estimators of \cite{CanTao07, FriMarMun12a,FriMarMun13,grasmair2015variational,haltmeier2017variational,Nem85}.

 \psfrag{B}{$f$}
\psfrag{Q}{$Q$}
\psfrag{a}{$(x,\xi)$}
\psfrag{transducer array}{linear/planar detector array}
\begin{figure}
\centering
\includegraphics[width=0.7\textwidth]{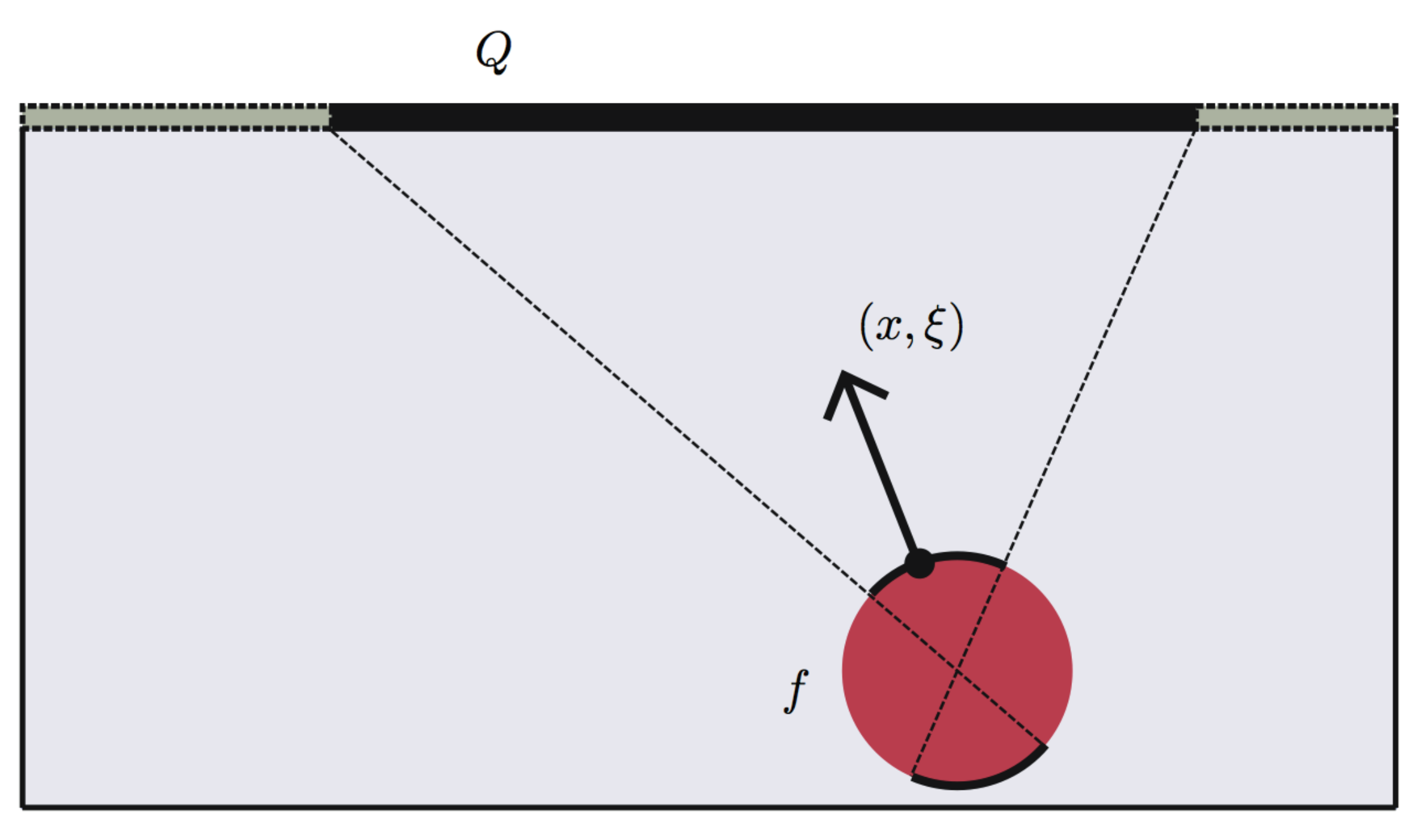}
\caption{\textsc{Limited view problem.} If data is are collected with
finite aperture $Q$, not all features of $f$ can be stably recovered.
Only singularities  $(x,\xi)$ can be reconstructed for  which $x + \R \xi$ intersects $Q$.}
\label{fg:setup}
\end{figure}

\section{Numerical implementation}
\label{sec:num}

In this section, we provide algorithms for the calculation of the vaguelette transform and the corresponding WVD estimator. Moreover, we consider a hybrid version of the WVD estimator that combines wavelets and TV regularization and discuss its implementation. We also present some numerical examples.
Throughout the following $(\wwave_\la)_{\la\in \La}$ denotes an orthonormal wavelet basis of $L^2(\R^d)$ and $(u_\la)_{\la\in \La}$ is the  corresponding orthogonal vaguelette basis of $\ran{\Ao} \subseteq L^2(\R^d)$ defined by $u_\la  \coloneqq \Ao \wwave_\la$ (see Section \ref{sec:wvd}).

\subsection{Practical aspects}

The considered noisy data model \eqref{eq:noisy} assumes that the data can be collected on the whole hyperplane $\partial \half_+ $.  However, this is not feasible in practice since the data can be collected only on a finite  subset.
In this subsection we  discuss the effects of partial (or limited view) data and discretization.

First, we address the limited data issue. In the considered imaging setup, the data is collected on a subset $Q \times [0,T]$ where
$Q \subseteq \R^{d-1}$  is the finite measurement aperture (see~Figure \ref{fg:setup})
and $T \in(0, \infty)$  the maximal measurement  time.  Such   partial or limited view data can  be  modeled by
\begin{equation}\label{eq:data-part}
    g(\xx, t)
    =
    \kl{\chi_{Q \times [0,T]} \, \Ao \signal} (\xx, t) \,,
\end{equation}
where $\chi_{Q \times [0,T]}$ denotes the characteristic function of $Q \times [0,T]$.
Using partial data $g$,  only certain features of $h$ can  be reconstructed in a stable way, see \cite{barannyk2015artifacts,frikel2015artifacts} and Figure \ref{fg:setup}. Consequently, the practical problem of reconstructing $h$  from partial data, is a limited data problem and therefore severely ill-posed.
It is therefore common to incorporate a-priori information into the reconstruction and so to regularize the reconstruction. In this work, we are doing it in two ways: First, we incorporate wavelet sparsity assumptions. This is what the WVD estimator does, which is implemented by
evaluating $\signal^\dag = \sum_{\la \in \La} \inner{g,  u_\la}_{L^2} \wwave_ \la $. The partial reconstruction $f^\dag$ is a ``good'' approximate reconstruction  for $f$,  since it recovers all visible boundaries of $f$ correctly and can be evaluated stably.   Second, we combine wavelet sparsity with TV regularization by using the using the hybrid estimator  \eqref{eq:dantzig} to impose even more regularization.

Another practical restriction is that only a finite number of samples
of the pressure can be measured.
Assuming equidistant sampling and  limited view data
$g$  as in \eqref{eq:data-part} with $Q = [0,T]^{d-1}$,
the actual sampled  data are  given by
\begin{equation} \label{eq:samp}
    \gn[n, m] \coloneqq
    g (  n \Delta_N, m \Delta_M)
    + \noise[n, m] \,,
    \quad (n,m) \in \set{1, \dots, N}^{d-1} \times \set{1, \dots, m}    \,.
\end{equation}
Here  $N$ and $M$ are natural numbers, $\Delta_N := X/N$  and $\Delta_M := T/M$ are  the  sampling step sizes and $\noise[n, m]$ decribes
the noise in the data. Using Shannon sampling theory it ca be shown that
\eqref{eq:samp} is correctly sampled for $\Delta_N = \Delta_M \leq \pi /\Omega$,
where $\Omega$  is the essential bandwidth of $f$;
 see~\cite{HalSchZan09b}.
 The precise analysis of discretization effects on the considered vaguelette estimators is  beyond the  scope of  this paper.

\subsection{Implementation of the vaguelette transform}

Analogously to the wavelet transform, we can define the vaguelette
transform of $ g \in L^2(\R^d)$ corresponding to the operator $\Ao$
by  $\VT g  \coloneqq \kl{\inner{u_\la, g}_{L^2}}_{\la \in \La}$.
From the representation $u_\la  = \Ao \wwave_\la$ and the definition of vaguelette transform we have
\begin{equation}
\forall g \in L^2(\R^d) \;
\forall \la \in \La \colon \quad
(\VT g)(\la)
	 =\inner{\Ao \wwave_\la,  g}_{L^2}
	 =\inner{\wwave_\la,  \Ao^\ast g}_{L^2} \,.
\end{equation}
Hence,  the   vaguelette transform can be computed  by first applying the
adjoint $\Ao^*$  to the data $g$ and then calculating  the wavelet
transform of $\Ao^* g$.

\begin{algorithm}[h]
\caption{Discrete vaguelette transform} \label{alg:v}
\begin{algorithmic}[1]
\medskip
\State
Compute $\Bn \gn \in \R^{N^{d-1} \times M}$
with $\Bn \gn  [n,m] \simeq \Ao^* g (n \Delta_N,\Delta_M)$

\State
Compute the discrete wavelet transform  $\Wn \Bn \gn$
\medskip
\end{algorithmic}
\end{algorithm}

Assuming discrete data $\gn \in \R^{N^{d-1} \times M}$
of the form $ \gn[n, m] = g (  n \Delta_N, m \Delta_M)$, the discrete vaguelette transform
can efficiently be computed by Algorithm~\ref{alg:v}.
Both the steps in this algorithm are well known
and numerically efficient. For the first step we use a numerical
approximation  of $\Ao^*$ by numerically  implementing~\eqref{eq:wave-ad} with a filtered backprojection (FBP) algorithm (see \cite{BurBauGruHalPal07,burgholzer2007exact}). For evaluating the second step we use the implementation of the fast discrete wavelet transform provided by the \textsc{Matlab}  function {\tt wavedec2}.

\subsection{Implementation of the reconstruction algorithms}
\label{sec:implementation}

For the numerical experiments, that will be presented in the next section,
we implement the   WVD soft-thresholding estimator defined by \eqref{eq:wvd-soft}
as well as a hybrid vaguelette-TV approach defined by \eqref{eq:dantzig} with $\rf(f) =  \norm{f}_{\rm TV}$ being the TV norm of $f$.  In both cases we choose $q$ as a half of the so called universal threshold  $\sigma \sqrt{2 \log (N^dM)}$  that can be derived  from
extreme value theory~\cite{Don95a,haltmeier2013extreme}.

The implementation of the WVD soft-thresholding estimator is summarized in Algorithm~\ref{alg:soft}. It is based on the discrete vaguelette transform that was presented in Algorithm~\ref{alg:v}.

\begin{algorithm}[h]
\caption{WVD soft-thresholding  estimator.} \label{alg:soft}
\begin{algorithmic}[1]
\medskip
\State
Compute vaguelette  coefficients $\cn = \Vn \gn$
by applying Algorithm~\ref{alg:v}

\State
Apply soft thresholding to $\cn$ using threshold $q$

\State
Apply the inverse wavelet transform
\medskip
\end{algorithmic}
\end{algorithm}

Given discrete data, the hybrid vaguelette-TV approach
 can be written in the form
\begin{equation} \label{eq:tv}
	\min_{\fn}
	\norm{\fn}_{\rm TV}
	\quad
	\text{ such that }
	\norm{ \Wn ( \Bn \gn -  \fn )}_\infty
	\leq   q
	\,,
\end{equation}
where $\Bn$ denotes the discrete FBP operator.
One recognizes that \eqref{eq:tv} can be implemented by
applying a hybrid  wavelet-TV denoising algorithm  to
$ \Bn \gn$.  Following~\cite{FriMarMun12b},
we rewrite the optimization problem~\eqref{eq:tv} in the form
\begin{equation} \label{eq:mre-aux}
	\min_{\fn, \vn}
	 \norm{\fn}_{\rm TV}
	 + i_{ \{ \norm{ \Wn \vn }_\infty \leq  \threshp\} }
	 \quad \text{ such that }
	 \fn  + \vn = \Bn \gn \,,
\end{equation}
and introduce the associated augmented
Lagrangian operator
\[ \mathcal{L}_c \kl{\fn, \vn, \mu}
\coloneqq
\norm{\fn}_{\rm TV}
+ i_{ \{ \norm{ \Wn \vn }_\infty \leq \threshp\} }
+ \inner{\mu, \fn  + \vn - \Bn\gn}
+ \frac{1}{2c} \norm{ \fn  + \vn - \Bn\gn}_2^2.\]
Then, \eqref{eq:mre-aux}  can be solved  by
the alternating direction  method of multipliers
(ADMM), introduced in \cite{GabMer76,GloMar75},
which alternatingly    performs minimization steps with respect to
$\fn$ and $\vn$  and maximization steps with respect to $\mu$.
The resulting ADMM algorithm for solving  \eqref{eq:tv} is
summarized in Algorithm~\ref{alg:admm-tv}.

\begin{algorithm}[h]
\caption{Hybrid vaguelette-TV approach}
\label{alg:admm-tv}

\begin{algorithmic}[1]
\medskip
\State
$\fn_0 = \vn_0 = \mu_0 = 0$

\For{$k=0,1,\dots, N_{\rm iter}$}
\State
$\fn_{k+1}
\in \argmin{\frac{1}{2}
	\norm{ \fn - \kl{ \Bn\gn - \vn_k - c \mu_k}}_2^2
	+ c \norm{\fn}_{\rm TV}}$

\State
$\vn_{k+1}
 =
 q \kl{\Bn\gn  - \fn_{k+1} - c \mu_k}
 / \max \set{q, \snorm{\Bn\gn  - \fn_{k+1} - c\mu_k}_2}
$

\State
$\mu_{k+1} \coloneqq \mu_k + c^{-1} \kl{\fn_{k+1} +  \vn_{k+1} - \Bn\gn }
$
\EndFor
\medskip
\end{algorithmic}
\end{algorithm}

For performing the $\fn$-update in Algorithm \ref{alg:admm-tv}
we have to solve the  unconstraint total variation minimization problem
$\frac{1}{2}
\norm{ \fn - \kl{ \Bn\gn - \vn_k - c \mu_k}}_2^2
+ c \norm{\fn}_{\rm TV}$. For this purpose we use
Chambolle's dual projection algorithm \cite{Cha04}.

\begin{figure}[tbh!]
\centering
\includegraphics[width=\textwidth, trim={ 1.7cm 1cm 1.7cm 0.3cm}]{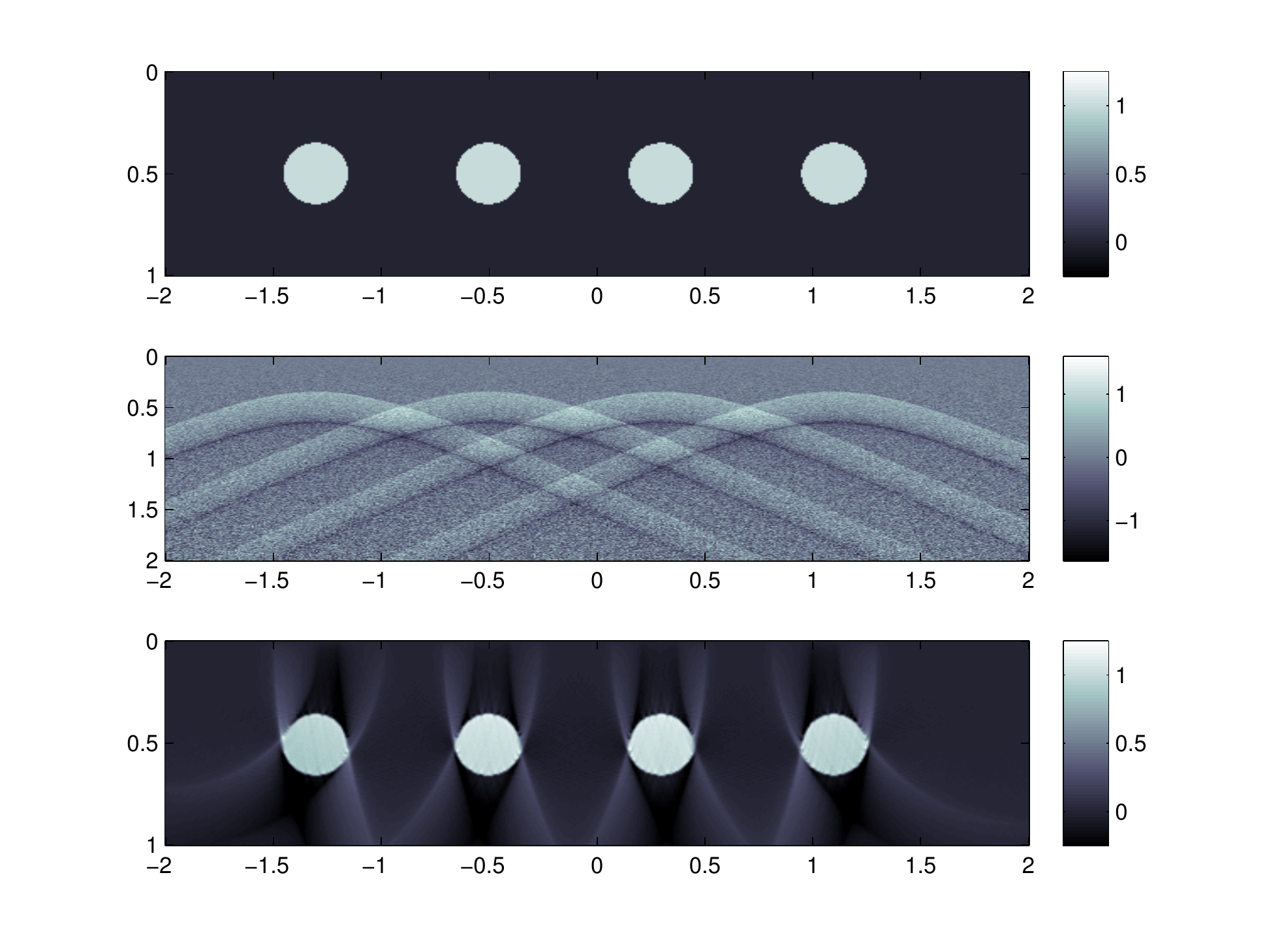}
\caption{\textsc{Phantom, data and reconstruction from (noise free) partial data.} One
notices the typical limited view artifacts in the
form of smearing out of almost vertical boundaries.\label{fig:fg}}
\end{figure}

\subsection{Numerical examples}

In this section we present a numerical example testing the
vaguelette soft-thresholding and the hybrid approach. For that purpose we consider
a simple phantom that consists  of the superposition of three
uniformly absorbing spheres as illustrated in the top image in
Figure~\ref{fig:fg}. The data $\gn = \An \fn $ are  computed numerically using an  implementation according to \eqref{eq:wave}. To simulate data errors, we added i.i.d Gaussian
 noise with standard deviation $\sigma = 0.25$. The relative $\ell^2$-error   in the  noisy data $\gn^\delta$ is  $\norm{\gn^\delta -\gn}_2 / \norm{\gn}_2 = 1.05$.

\begin{figure}[tbh!]
\centering
\includegraphics[width=\textwidth, trim={ 1.7cm 1cm 1.7cm 0.3cm}]{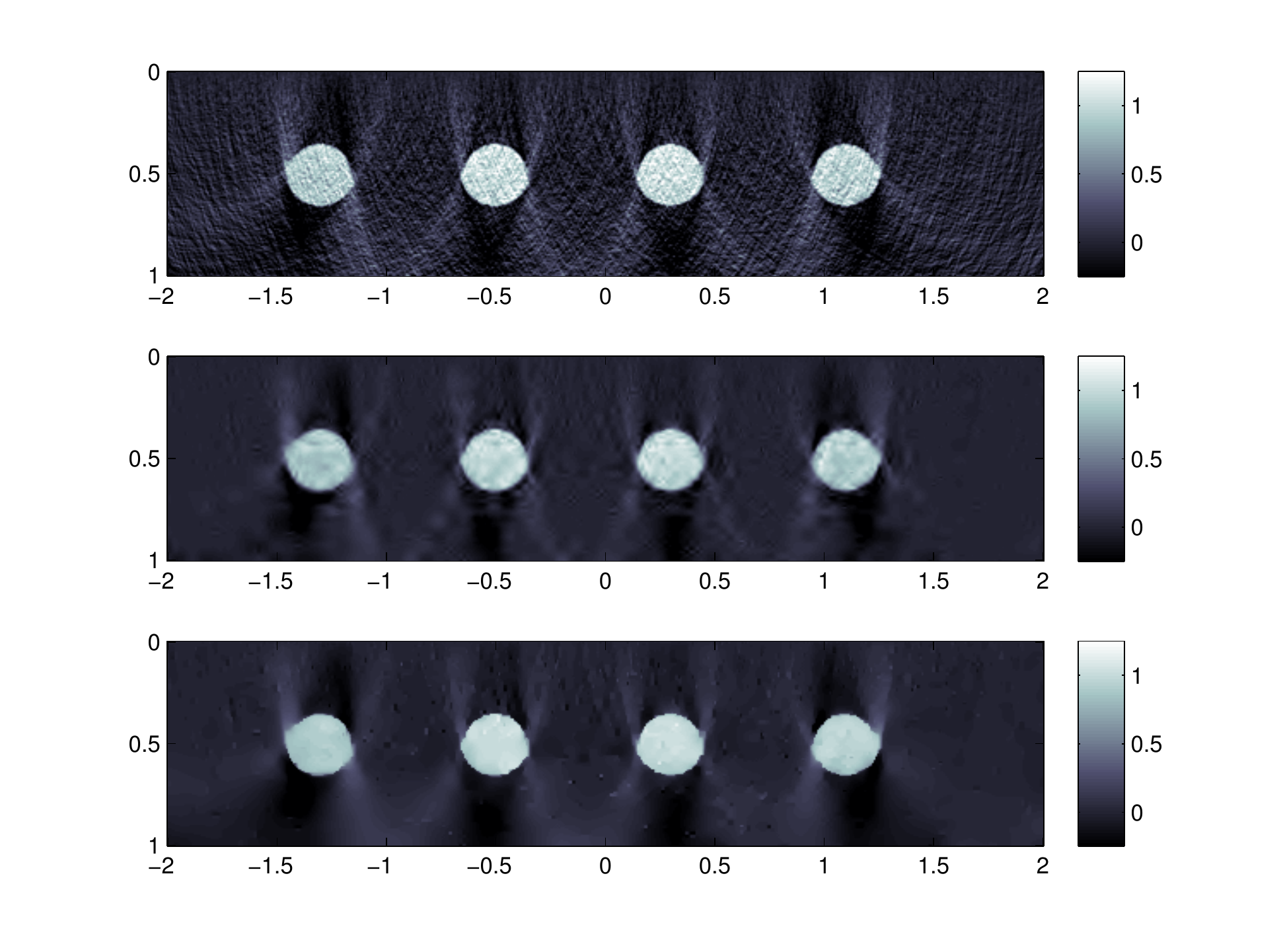}
\caption{\textsc{Reconstructions from noisy data.}\label{fig:rec} }
\end{figure}

\begin{figure}[tbh!]
\centering
\includegraphics[width=\textwidth, trim={ 1.7cm 1cm 1.7cm 0.3cm}]{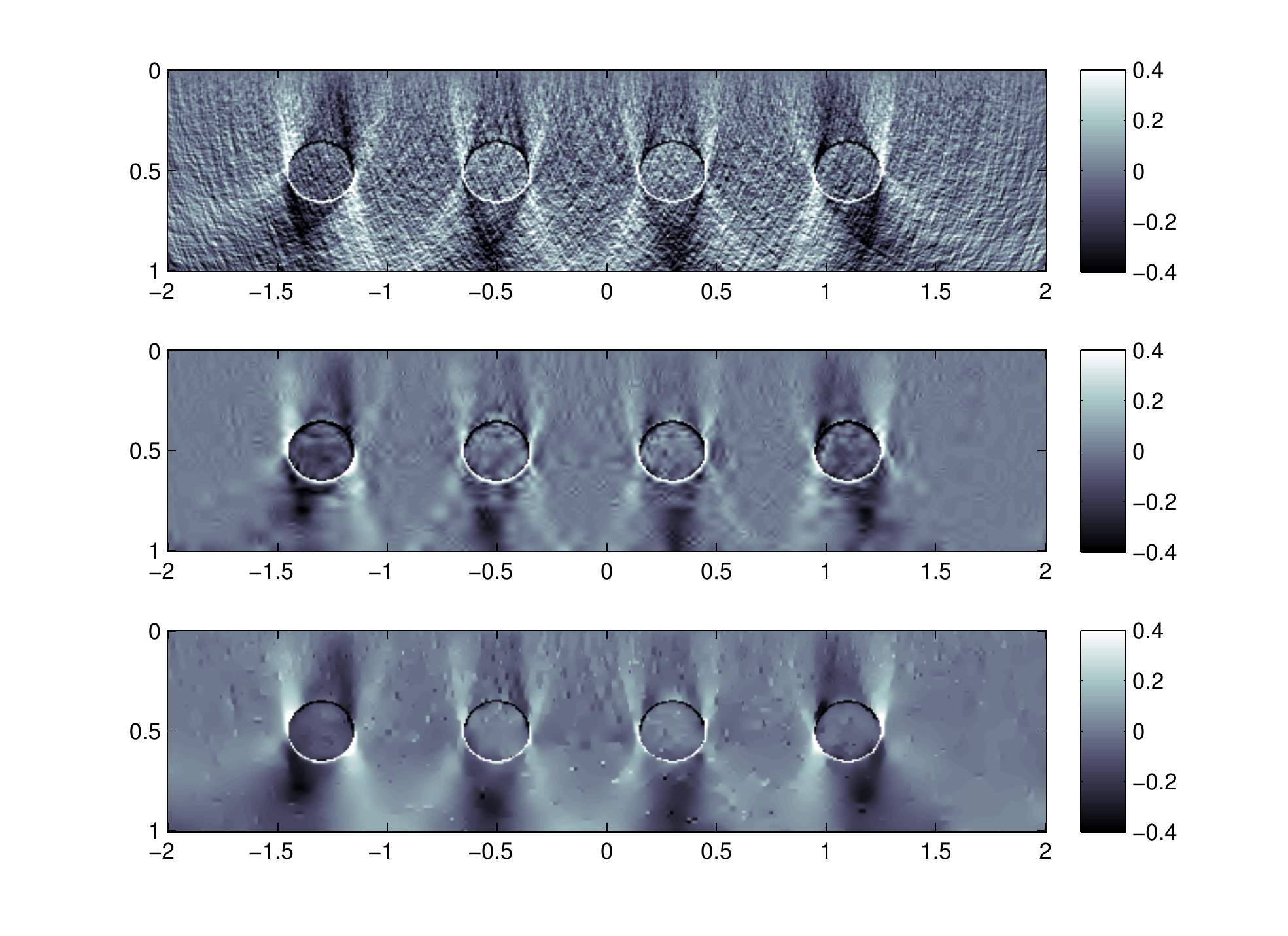}
\caption{\textsc{Differences between reconstructions from noisy data and the original phantom.}\label{fig:err} }
\end{figure}


Results of our reconstruction from noisy data are shown Figure~\ref{fig:rec}.
The top image shows the non-regularized reconstruction  $\fn_{\rm FBP}^\delta = \An \gn^\delta$, the middle image the vaguelette-thresholding reconstruction $\fn_{\rm WVD}^\delta$ and the bottom image the  reconstruction with the hybrid
vaguelette-TV method $\fn_{\rm hybrid}^\delta$.  Figure~\ref{fig:err} shows corresponding difference images between the reconstructions and the true phantom  $\fn$. One clearly observes that the vaguelette estimators significantly reduce the error in the reconstruction. This can be quantified by the relative
$\ell^2$-errors, which are equal to
$\snorm{\fn_{\rm FBP}^\delta -\fn}_2 / \snorm{\fn}_2 = 0.66$
for the unregularized reconstruction,
$\snorm{\fn_{\rm WVD}^\delta -\fn}_2 / \snorm{\fn}_2 = 0.38$
for the vaguelette thresholding estimator
$\snorm{\fn_{\rm hybrid}^\delta -\fn}_2 / \snorm{\fn}_2 = 0.42$
for the hybrid estimator. The relative
$\ell^2$-error for  the hybrid reconstruction is slightly larger than for the vaguelette-thresholding estimator. This is reasonable since the  vaguelette-thresholding estimator minimizes
the $\ell^2$-norm (see \eqref{eq:wvd-const}) among all potential reconstructions  $\fn$ that are compatible with  the data, whereas the hybrid estimator minimizes the total variation.
However, in some more appropriate error measure  the  hybrid reconstruction may outperform the thresholding estimator.
Further note that all reconstruction methods contain
some limited data artifacts, which cannot removed completely  by
wavelet methods \cite{barannyk2015artifacts,frikel2015artifacts,HalSchZan09b,nguyen2015strong,stefanov2013curved}.

\section{Conclusion}

In this paper we developed a regularization framework using wavelet sparsity in photoacoustic tomography.
For that purpose we derived wavelet-vaguelette (WVD) decompositions (see Theorem~\ref{thm:wvd}) and an easy but
efficient implementation of the corresponding  vaguelette transform (see~Algorithm \ref{alg:v}). Using the WVD we
derived an explicit formula for minimizing the sparse Tikhonov functional that can be implemented without any
iterative reconstruction procedure (see Algorithm~\ref{alg:soft}).  The considered regularization
approach has been shown to provide optimal  error estimates in the deterministic as well as in the stochastic setting (see Theorem~\ref{thm:random}).
In order to account for wavelet artifacts we also developed hybride  regularization methods combining wavelet sparsity with
total variation (see Algorithm~\ref{alg:admm-tv}). Numerical results  demonstrate the feasibility and
efficiency of our reconstruction approaches. Future work will be done to extend our approach to more general measurements
geometries in PAT and also to different Radon type  inverse problems.

\appendix

\section{Orthogonal wavelets}
\label{app:wavelets}
We recall some basic facts about orthogonal wavelets as we need
them for our purpose (in particular in Section \ref{sec:wvd}). For task of function
estimation, wavelets are known to sparsely represent many signals and, hence,
they can be used to effectively encode prior information. Another useful property of
wavelets consists in the ability to characterize several classical
smoothness measures (eg. Sobolev  and Besov norms)
in terms of the decay properties of wavelet coefficients. We will also use that
property in Section \ref{sec:noisy}. For a detailed introduction to wavelets we refer to \cite{Coh03,Dau92,Mal09}.

\subsection{One-dimensional wavelets}

We first briefly recall the basic definitions and notations of orthonormal wavelet bases (wavelet ONB) in one spatial dimension, which will be then extended to higher dimensions.

The construction of a wavelet ONB is based on the concept of a multiresolution analysis (MSA), which is  given by a sequence subspaces $(V_j)_{j\in\Z}$ in $ L^2\kl{\R}$ that satisfy the following requirements (see \cite{Coh03,Dau92,Mal09}):
\begin{itemize}[topsep=0em]

\item For all $j\in\Z$ it holds that $V_j\subseteq V_{j+1}$.
\item
The union $\bigcup_{j\in \N}\V_{j}$ is dense in
$L^2\kl{\R}$.
\item $\bigcap_{j\in \Z}\V_{j}=\sparen{0}$ .
\item For every $j \in \N$, we have
$f \in \V_j \iff f(\tfrac{\edot}{2}) \in \V_{j+1}$.
\item There is a function $\wscale\in L^2\kl{\R}$ such that the translates $(\wscale(\cdot-k))_{k\in\Z}$ constitute an ONB of the scaling space $V_0$.
\end{itemize}
The function $\wscale$ is called \emph{scaling function} and the spaces $V_j$ are called \emph{scaling (or approximation) spaces at scale j}. To each scaling space $V_j$, one can associate a \emph{wavelet (or detail) space $W_j$}, that are defined to be the orthogonal complements of $V_j$ in $V_{j+1}$, i.e., $V_{j+1}=V_j\oplus W_j$.
Because of the above properties it holds that, for each $j\in \Z$ and $t\in\R$, the functions
\begin{equation*}
t\mapsto\wscale_{j,k}\kl{t}
 \coloneqq 2^{j/2}
\wscale \kl{2^{j} t - k}
\end{equation*}
constitute an ONB for the scaling space $V_j$. One can also show that from the existence of the scaling function it follows that there exists a so-called \emph{generating wavelet (or mother wavelet)}  $\wwave \in L^2\kl{\R}$ such that,  for each $j\in \Z$ and $t\in\R$, the functions
\begin{equation*}
\wwave_{j,k}\kl{t}
 \coloneqq 2^{j/2}
\wscale \kl{2^{j} t - k}
\end{equation*}
constitute an ONB for the spaces $W_j$. Hence, we have that, for every $j\in\Z$, the following
mappings are bijections:
\begin{align*}
\V_{j}  \to \ell^2(\Z)
&\colon
\signal \mapsto
\kl{\inner{\wscale_{j,k},\signal} \mid  k \in \Z} \,,
\\
\W_{j}  \to \ell^2(\Z)
&\colon
\signal \mapsto
\kl{\inner{\wwave_{j,k},\signal} \mid  k \in \Z } \,.
\end{align*}

The above constructions provide the following decompositions of the signal space $L^2\kl{\R}$ into the  sum of the scaling
spaces  $\V_j$ and the wavelet spaces $\W_j$, or into a sum of only wavelet spaces:
\begin{equation*}
	L^2\kl{\R} = V_0\oplus \bigoplus_{j\geq 0} W_j = \bigoplus_{j\in\Z} W_j.
\end{equation*}
From these decomposition we immediately get the following decompositions of signals $f\in L^2\kl{\R}$:
\begin{align*}
	f &= \sum_{k\in\Z} \langle f, \wscale_{0,k}\rangle \wscale_{0,k} + \sum_{j\geq 0}\sum_{k\in \Z} \langle f, \wwave_{j,k}\rangle \wwave_{j,k}\\
	&=\sum_{j\in\Z}\sum_{k\in \Z} \langle f, \wwave_{j,k}\rangle \wwave_{j,k}.
\end{align*}
The coefficients of the above decomposition are called the \emph{wavelet and scaling coefficients of $f$}, respectively, and the corresponding mapping that maps $f$ to those coefficients is called the \emph{wavelet transform of $f$}. A detailed construction of orthogonal (and biorthogonal) wavelet systems together with  many interesting  details may be found, for example, in~\cite{Coh03,Dau92,Mal09}.

\subsection{Wavelets in higher dimension}

Wavelet bases in higher dimensions can be defined  by taking
tensor  products of the one-dimensional wavelet and scaling functions.
As a concrete example consider  $L^2(\R^2) \cong L^2(\R) \otimes L^2(\R)$ and let $\wwave, \wscale \colon \R \to \R$ be the generating wavelet and scaling functions, respectively. Then, an orthogonal wavelet basis of $L^2(\R^2)$
is defined by translates and scaled versions  of  the (tensor product) functions
\begin{align*}
\wwave^{(1)} \colon &(x,y) \in \R^2 \mapsto \wwave(x) \wscale(y) \\
\wwave^{(2)} \colon &(x,y) \in \R^2  \mapsto \wscale(x)      \wwave(y) \\
\wwave^{(3)} \colon &(x,y) \in \R^2\mapsto  \wwave(x)  \wwave(y) \,.
\end{align*}
The corresponding scaling functions at scale $j\in\Z$ of two variables are defined as translated and  scaled
versions of \[\wscale^{(0)}(x,y) \coloneqq  \wscale(x) \wscale(y).\]

Analogously to the above construction in two dimensions, one can define a wavelet ONB in $\R^d$ by a tensor product construction. In this case we get $2^d-1$ different generating wavelets $ \wwave^{(\eps)}$ and one generating
 scaling function $ \wscale^{(0)}$. Therefore, a wavelet ONB in $\R^d$ is system
$(\wwave_\la)_{\la \in \Lambda}$  in $L^2(\R^d)$, where the index $\la =(j,k,\eps)$ consists of three parameters, the  scale index $j \in \Z$, the location index $k \in \Z^d$,
and the  orientation index $\eps \in \set{1,2,\dots, 2^d-1}$. The
wavelet basis elements are again defined as dilates and translates of the generating wavelets $\wwave^{(\eps)}$:
\begin{equation}
	\wwave_{j, k,\eps} (x)
	=  2^{-jd/2} \, \wwave^{(\eps)}(2^j x - k)  \,.
\end{equation}
If we also add the index set $\La_{-1}  \coloneqq \set{0} \times \Z^d$
and consider the scaling functions defined by \[\wscale_{0, k}^\eps (x)
=  2^{-jd/2} \wscale^{(0)}(2^j x - k),\] the we get similar decompositions to the one-dimensional case:
\begin{align*}
	f &= \sum_{k\in\Z^d} \langle f, \wscale_{0,k}\rangle \wscale_{0,k} + \sum_{j\geq 0}\sum_{k\in \Z^d}\sum_{\epsilon=1}^{2^d-1} \langle f, \wwave_{j,k,\epsilon}\rangle \wwave_{j,k,\epsilon}\\
	&=\sum_{j\in\Z}\sum_{k\in \Z^d}\sum_{\epsilon=1}^{2^d-1} \langle f, \wwave_{j,k,\epsilon}\rangle \wwave_{j,k,\epsilon}.
\end{align*}

For $\signal \in L^2(\R^d)$ we define the wavelet
transform as
\begin{equation}\label{eq:wt}
(\Wave \signal ) \colon \La \to \R  \colon
\la \mapsto (\Wave\signal)(\la) \coloneqq \inner{\wwave_\la,\signal} \,.
\end{equation}
Then the wavelet transform  $\Wave \colon  L^2\kl{\R^d}  \to \ell^2\kl{\La}
\colon  f \mapsto \Wave f$ is an isometric isomorphism with the
reproducing property $\signal = \Wave^* \Wave \signal$.

\subsection{Besov spaces}
Orthonormal wavelet bases can be used to characterize smoothness of functions in terms of the decay properties of wavelet coefficients.
In particular, wavelets provide a convenient and a numerically efficient way to characterize Besov spaces and calculate Besov smoothness of functions, see \cite{Coh03}.
Skipping the details, we here only mention that the Besov spaces $\B_{p,q}^s(\R^d)$ can be defined by the condition,
that the Besov norm
\begin{equation*}
\norm{\signal}_{\B_{p,q}^r}
=
\sqrt[p]{\sum_{j\in\N}
2^{jsq} \,
\norm{ \Wave \signal\kl{j,\edot}}_{\ell^p}^q}
\qquad
\text{ with } \quad
s = r +  \frac{d}{2} - \frac{d}{p}  \,,
\end{equation*}
is well defined and finite. We note that, the given definition of $\norm{\edot}_{\B_{p,q}^r}$
is an equivalent norm to the classical definition of Besov norms. The above characterization
holds as long as the generating wavelet  has  $m > r$ vanishing moments and is $m$-times continuously differentiable.

In Section \ref{sec:noisy}, we use the above characterization of Besov norms in terms of wavelet coefficients in order to evaluate the performance of our method for functions that lie in balls of Besov spaces $\B_{p,q}^s(\R^d)$ for some given norm parameters  $p$, $q \geq 1$ and smoothness parameter $r \geq 0 $.

\end{document}